\documentclass[11pt]{article}

\usepackage{amssymb}
\usepackage{amsthm}
\usepackage{amsmath}
\usepackage{latexsym}
\usepackage{tikz}
\newtheorem{theorem}{Theorem}[section]
\newtheorem{lemma}[theorem]{Lemma}
\newtheorem{cor}[theorem]{Corollary}

\newtheorem{remark}[theorem]{Remark}

\makeatletter \@addtoreset{equation}{section} \makeatother

\newcommand{\R}{\mathbb{R}}
\newcommand{\beq}{\begin{equation}}
\newcommand{\eeq}{\end{equation}}
\newcommand{\eps}{\varepsilon}

\def\bint{\mathop{\mathrel{\nwarrow}\!\!\!\!\!\!\mathrel{\int}}}
\def\fint{\mathop{\mathrel{\nearrow}\!\!\!\!\!\!\mathrel{\int}}}

\begin{document}

\title{Singularity formation for the
compressible Euler equations with general pressure law}

\author{Hualin Zheng\\
	\\
 \small{Mathematical Sciences Center, Tsinghua University,Beijing, 100084, P. R. China;}\\
\small{School of Mathematics, Georgia Institute of Technology, 686 Cherry Street, 
 	Atlanta,}\\
 \small{GA 30332-0160 USA.}\\
 \small{Email:~~zhenghl12@mails.tsinghua.edu.cn;\quad
 	hzheng62@math.gatech.edu}
}

\date{}
\maketitle
\begin{abstract}
In this paper, the singularity formation of classical solutions for the compressible Euler equations with general pressure law is considered.
The gradient blow-up of classical solutions is shown without any smallness assumption
by the delicate analysis on the decoupled Riccati type equations. The proof also relies on a new estimate
for the upper bound of density.
\end{abstract}

\section{Introduction}
We consider the one dimensional compressible Euler equations in the Lagrangian coordinates:
\begin{eqnarray}\label{equ-Euler}
\begin{cases}
\tau_t - u_x = 0,\\
u_t+p_x=0,\\
\left(e+\frac{u^2}{2}\right)_t + (up)_x = 0,
\end{cases}
\end{eqnarray}
where $x$ is the space variable, $t$ is the time variable, $u$ is the velocity,
$\rho$ is the density, $\tau=\rho^{-1}$ is the specific volume, 
$p$ is the pressure,  $e$ is the internal energy.
Due to the second law of thermodynamics, $\tau$, $p$ and $e$ are not independent,
the relation within which is determined by the state equation(\textit{c.f.} \cite{Dafermos}).
Normally, another physical quantity entropy $S$ is considered, which formulates the
state equation as $p=p(\tau, S)$.
For $C^1$  solutions, the third equation of \eqref{equ-Euler} is equivalent to the conservation of entropy (\textit{c.f.} \cite{smoller}):
\begin{equation}\label{equ-entropy}
S_t =0.
\end{equation}
Apparently, \eqref{equ-entropy} shows that $S$ is just the function of $x$.
And, the general pressure law we consider in this paper is
\begin{equation}\label{equ-p}
p=p(\tau, S)=p(\tau, S(x)).
\end{equation}
Then the system \eqref{equ-Euler} becomes
\begin{eqnarray}\label{system-general-pressure}
\begin{cases}
\tau_t - u_x = 0,\\
u_t + p(\tau, S(x))_x = 0.\\
\end{cases}
\end{eqnarray}
We consider the calssical solution of initial value problem for \eqref{system-general-pressure} with initial data
$$
\tau(x,t=0)=\tau_0(x), \quad 
u(x,t=0)= u_0(x).
$$ 

Compressible Euler equations
is one of the most important physical models for systems of hyperbolic conservation laws.
It is well known that shock waves are typically formed in finite time
and the analysis on the system is difficult because of the lack of regularity. 
The singularity
formation for both the small initial data problem and the large initial data problem
has long been a very serious issue for the systems of conservation laws.
The well-posedness theory for systems of hyperbolic conservation laws could be found in \cite{Bressan, Dafermos, Glimm,TW}.

When initial data is small, the singularity formation has been well studied for decades.
Lax \cite{lax2} proved that singularity forms in finite time for the general  systems of strictly hyperbolic conservation laws with two unknowns with some initial compression.
For general systems of conservation laws, 
\cite{Fritz John, li-zhou-kong, Li daqian, Liu1} provide fairly complete results for small data. 
Specifically, these results prove that the shock formation happens in finite time in any truly nonlinear characteristic field if the initial data includes compression.

However, the large data singularity formation theory has been finally established in very recent papers
\cite{CPZ, chen-young-zhang} for isentropic Euler equations with $\gamma$-law pressure ($p=K_1\tau^{-\gamma}$) and
the full compressible Euler equations of polytropic ideal gas ($p=K_2e^{S/c_\tau}\tau^{-\gamma}$),
where $K_1$, $K_2$ are positive constants and $\gamma>1$ is the adiabatic gas constant.
The key point in proving the finite time shock formation for large solution
is to have sharp global upper and lower bounds of the density.
More precisely, if we restrict our consideration on singularity formation for full compressible Euler equations, the uniform upper bound of density is needed for any $\gamma>1$, but
the time-dependent lower bound of density is needed only for the most physical case $1<\gamma<3$ (\textit{c.f.} \cite{G3}).
The uniform upper bound on density for $\gamma$-law pressure case has been found by
\cite{chen-young-zhang} which directs to a resolution of the shock formation when 
$\gamma\geq 3$. The singularity formation problem when $1<\gamma<3$ was finally resolved by \cite{CPZ}, in which the authors proved a crucial time-dependent lower bound
estimate on density lower bound. Later on, the time-dependent lower bound of density is improved to its optimal order $O(1/t)$ in \cite{geng chen lower bound}.

Nevertheless, for the full compressible Euler equations with general pressure law, the singularity formation results for non-isentropic case are still not satisfied when the smallness assumption on the initial data is removed. In fact, a complete finite time gradient blow-up result has been showed in \cite{CPZ} when entropy $S$ is a given constant.
Furthermore, \cite{chen-young} provides a singularity formation result for the non-isentropic general pressure law case. Unfortunately, in \cite{chen-young},
there are still several a priori conditions on the pressure function which are not automatically
satisfied for the gas dynamics. The target of this paper is to establish a better singularity formation result on non-isentropic Euler equations without such kind of a priori assumptions. 
The key idea is to establish a uniform upper bound estimate on density, which was lack for general pressure law case previously. In this case, the lower bound of density is redundant. 
Our proof relies on the careful study on the 
decoupled Riccati type ordinary differential equations on gradient variables which was
provided in \cite{chen-young}. 
Using our new estimates, we can get the constant lower bound on coefficients of the Riccati type equations, and the quadratic nonlinearity implies the derivatives must blow-up in finite time.

%%%%%%%%%
%%%%%%%%%
%%%%%%%%%

Through out this paper, we need to propose the following assumptions on the pressure: there exists a positive function $m=m(S)$, positive constants $A$ 
, $k>1$, $k_1, k_2$ and $l_i$ $(i=1,2,\cdots, 8)$ such that, for $\tau\in(0, +\infty)$, 
\begin{eqnarray}
\textbf{(H1)}
&& p_\tau<0, \quad p_{\tau\tau}>0,\quad 
\lim\limits_{\tau\rightarrow 0} p(\tau) = + \infty, \quad
\lim\limits_{\tau\rightarrow + \infty} p(\tau) = 0,\label{relation-p1}\\
\textbf{(H2)}	
&&\int_0^1 \sqrt{-p_\tau} d\tau = + \infty,\quad\quad
\int_1^{+\infty} \sqrt{-p_\tau} d\tau < + \infty,\label{relation-p2}\\
\textbf{(H3)}	
&&	l_2 p c^{\frac{7}{2}} \leq pp_{\tau\tau} \leq l_1c^{\frac{7}{2}},
\label{assumption-p-ptautau-ptau}\\
&&	2(k-1)(-p_\tau)^2\geq k p p_{\tau\tau},\quad
(5+A) (p_{\tau\tau})^2 - 4 p_\tau p_{\tau\tau\tau} \geq 0,\label{p-tau-k-k-1-relation-A-ptau}\\
\textbf{(H4)}	
&&	\frac{m'(x)}{k_2 m} p \leq p_\mu \leq \frac{m'(x)}{k_1 m} p,
\quad
\frac{m'}{l_4 m} p_\tau \leq p_{\tau\mu} \leq \frac{m'}{l_3 m} p_\tau,\label{inequ-pmu} \\
&&	\frac{m'}{l_6 m} p \leq p_{\mu\mu} \leq \frac{m'}{l_5 m} p,
\quad
\frac{m'}{l_8 m} p_{\tau\tau} \leq p_{\tau\tau\mu} \leq \frac{m'}{l_7 m} p_{\tau\tau}. \label{inequ-ptaumu-pmumu-inequ-ptaotaomu}
	\end{eqnarray}
Here, the sound speed is 
\begin{equation}\label{equ-c}
c=\sqrt{-p_\tau(\tau, S)},
\end{equation}
and $p_\mu=\frac{\partial p(\tau,S(x))}{\partial S}S'(x)$.
\begin{remark}
$(\textbf{H1})$ is physically motivated for classical hydrodynamics  (\textit{c.f.} \cite{courant,menikoff-plohr}).
$(\textbf{H2})$ is the sound speed condition. $(\textbf{H3})$ and $(\textbf{H4})$ are
the nonlinearity conditions.
\end{remark}
Now, we introduce the derivatives combination of $(\tau, u)$:
\begin{eqnarray}\label{defi-y-q}
\begin{split}
y:=\sqrt{c} (u+h)_x + \frac{p_\mu}{\sqrt{c}}-I, \quad \quad
q:=\sqrt{c} (u-h)_x - \frac{p_\mu}{\sqrt{c}}+I,
\end{split}
\end{eqnarray}
where
\begin{equation}\label{equ-h}
h=\int_\tau^{+\infty}\sqrt{-p_\xi(\xi, S(x))}d\xi, \quad
I = \int_{h_0}^h \frac{\sqrt{c}}{2}\left(\frac{p_\mu}{c}\right)_h dh,
\end{equation}
and $h_0$ is a constant. 

Under the above assumptions, we can present the main theorem of this paper.
\begin{theorem}\label{main-thm}
Suppose that \textbf{(H1)}-\textbf{(H4)} are satisfied,
assume that 
the initial entropy $S=S(x)$ is $C^1$, finite piecewise monotonic and has bounded total variation, if 
 $\left(\tau_0(x), u_0(x)\right)$ are $C^1$ functions, and there are positive constants $k_{01}$ and $k_{02}$ such that
	$$
	\|\left(\tau_0(x), u_0(x)\right)\|_{C^1} \leq k_{01}, \quad \tau_0(x)\geq k_{02},
	\quad \mbox{for} ~x\in\R.
	$$
Then, the solution of the Cauchy problem of \eqref{system-general-pressure} blows up in finite time, if 
\begin{equation}\label{inequ-y0-q0}
\inf_{x\in \mathbb{R}}\{y(x,0), q(x,0)\}<-N.
\end{equation}
Here $N$ is a positive constant depending
only on $k_{01}$, $k_{02}$ and initial entropy function.
\end{theorem}
This paper is organized as follows: in section 2, we introduce some notations and prove the properties
of the pressure $p$. In section 3, we obtain the $L^\infty$ boundedness of the Riemann 
invariants, and this gives the upper bound of density and wave speed. In section 4, 
we prove the finite time singularity formation by analysising the Riccati type equations.

\section{Notations and Preparations}
We denote the forward and backward characteristics by
$$
\frac{dx}{dt}=c \quad \mbox{and} \quad   \frac{dx}{dt}=-c,
$$
and the corresponding directional derivatives along the characteristics are
\begin{eqnarray*}
	\partial_+:=\partial_t+c\partial_x  \quad \mbox{and} \quad
	\partial_-:=\partial_t-c\partial_x.
\end{eqnarray*}
Then we can denote the Riemann invariants by
\begin{equation}\label{rs}
s:=u+h  \quad \mbox{and} \quad
r:=u-h.
\end{equation}
We can easily get the following system of $u$ and $h$ (\textit{c.f.}\cite{chen-young}):
\begin{equation}\label{equ-hup}
\begin{cases}
h_t + cu_x =0,\\
u_t + ch_x +p_\mu = 0.
\end{cases}
\end{equation}
Thus, direct calculation shows that
\begin{eqnarray}\label{partial-sr}
\begin{split}
&\partial_+s
=(u_t+ch_x)+(h_t+cu_x)
=-\frac{\partial p(\tau,S(x))}{\partial S}S'(x),\\
&\partial_-r
=(u_t+ch_x)-(h_t+cu_x)
=-\frac{\partial p(\tau,S(x))}{\partial S}S'(x).
\end{split}
\end{eqnarray}
Furthermore, \eqref{inequ-pmu} yields the following inequality:
\begin{equation}\label{s+r-}
-\frac{p}{k_1}\frac{m'(x)}{m(x)}\leq \partial_+s=\partial_-r \leq -\frac{p}{k_2}\frac{m'(x)}{m(x)}.
\end{equation}

Next, we will prove the property of the pressure $p$ which will play a vital role in the proof
of Theorem \ref{main-thm}.
\begin{lemma}
Under the assumptions of (\textbf{H1}) and (\textbf{H3}),
\begin{equation}\label{inequ-pcsr}
\frac{1}{2k} c(s-r)\leq p\leq \frac{1}{2} c(s-r),
\end{equation}
where $k>1$.
\end{lemma}
\begin{proof}
If we can prove that
$$
p\leq ch\leq kp,
$$
then, due to $s-r=2h$, we deduce \eqref{inequ-pcsr}.

For the first part, we know that $-p_\tau$ is monotone decreasing in view of $(\ref{relation-p1})_2$, so we have
\begin{eqnarray*}
	ch
	=\int_\tau^{+\infty}\sqrt{-p_\tau}\sqrt{-p_\xi}d\xi
	\geq \int_\tau^{+\infty} \sqrt{-p_\xi}\sqrt{-p_\xi}d\xi=p,
\end{eqnarray*}
where we have used \eqref{equ-c}, $(\ref{equ-h})_1$ and $(\ref{relation-p1})_4$, thus we get $p\leq ch$.

For the second part, if there is a constant $k>1$ such that
\begin{equation}\label{inequ-k-p-c-ptau}
k\left(\frac{p}{c}\right)_\tau \leq -\sqrt{-p_\tau},
\end{equation}
integrating both sides from $\tau$ to $+\infty$ with respect to the time variable yields 
$ch\leq kp$. Actually, direct calculation shows that
\begin{eqnarray*}
	\left(\frac{p}{c}\right)_\tau=\left(\frac{p}{\sqrt{-p_\tau}}\right)_\tau
	=-\sqrt{-p_\tau}+\frac{1}{2}pp_{\tau\tau}(-p_\tau)^{-\frac{3}{2}},
\end{eqnarray*}
this yields that \eqref{inequ-k-p-c-ptau} is equivalent to $(\ref{p-tau-k-k-1-relation-A-ptau})_1$.
Thus, we prove the result of this lemma.
\end{proof}
\begin{remark}
It is worth noticing that $p\leq ch$ is the direct conclusion of the property of $p$. So, 
the following property is natural
\begin{equation}\label{p-leq-ch}
\frac{p}{c}\leq h.
\end{equation}
Here $p$, $c$ and $h$ are defined by \eqref{equ-p}, \eqref{equ-c} and \eqref{equ-h}
respectively.
\end{remark}
\section{The $L^\infty$ boundedness of $s$ and $r$}
In this section, we will first prove the $L^\infty$ boundedness of the Riemann invariants $s$
and $r$. Based on this, we can get the boundedness of $|u|$ and $|h|$. Finally, the upper bound of density $\rho$ will be obtained,  which is crucial to gain the singularity formation. Notice that
\begin{equation}\label{m+-}
\partial_+m=cm'(x) \quad \mbox{and} \quad \partial_-m=-cm'(x).
\end{equation}
We will discuss from two aspects according to the sign of the derivative of $m(x)$:

(\textbf{i}) When $m'(x)\geq0$, using \eqref{s+r-} and \eqref{inequ-pcsr}, we have
$$
-\frac{c}{2k_1} \frac{m'(x)}{m(x)}(s-r)\leq
\partial_+s=\partial_-r
\leq -\frac{c}{2kk_2}\frac{m'(x)}{m(x)}(s-r),
$$
which means,
\begin{eqnarray*}\begin{split}
-\frac{1}{2k_1} \frac{\partial_+m}{m}(s-r)\leq&\partial_+s\leq-\frac{1}{2kk_2}\frac{\partial_+m}{m}(s-r),\\
\frac{1}{2k_1} \frac{\partial_-m}{m}(s-r)\leq&\partial_-r\leq\frac{1}{2kk_2}\frac{\partial_-m}{m}(s-r).
\end{split}\end{eqnarray*}
Introducing new variables
\begin{eqnarray}\label{defi-s1r1}
s_{11} = m^{\frac{1}{2k_1}}s,
\quad
r_{11} = m^{\frac{1}{2k_1}}r,
\quad
s_{12}=m^{\frac{1}{2kk_2}}s,
\quad
r_{12}=m^{\frac{1}{2kk_2}}r,
\end{eqnarray}
then
\begin{equation}\label{inequ-s1}
\partial_+s_{11}\geq\frac{1}{2k_1}\frac{\partial_+m}{m}r_{11},
\quad
\partial_+s_{12}\leq\frac{1}{2kk_2}\frac{\partial_+m}{m}r_{12},
\end{equation}
\begin{equation}\label{inequ-r1}
\partial_-r_{11}\geq\frac{1}{2k_1}\frac{\partial_-m}{m}s_{11},
\quad
\partial_-r_{12}\leq\frac{1}{2kk_2}\frac{\partial_-m}{m}s_{12}.
\end{equation}

(\textbf{ii}) When $m'(x)\leq 0$, we similarly have
\begin{equation}\label{inequ-s2}
\partial_+s_{21}\geq\frac{1}{2kk_1}\frac{\partial_+m}{m}r_{21},\quad
\partial_+s_{22}\leq\frac{1}{2k_2}\frac{\partial_+m}{m}r_{22},
\end{equation}
\begin{equation}\label{inequ-r2}
\partial_-r_{21}\geq\frac{1}{2kk_1}\frac{\partial_-m}{m}s_{21},\quad
\partial_-r_{22}\leq\frac{1}{2k_2}\frac{\partial_-m}{m}s_{22},
\end{equation}
where
\begin{eqnarray}\label{defi-s2r2}
s_{21} = m^{\frac{1}{2kk_1}}s,\quad r_{21} = m^{\frac{1}{2kk_1}}r,\quad
s_{22} = m^{\frac{1}{2k_2}}s,\quad r_{22} = m^{\frac{1}{2k_2}}r.
\end{eqnarray}
		
According to the assumptions on the initial entropy in Theorem \ref{main-thm}, we have
\begin{equation}\label{v}
V:= \frac{1}{2c_\tau}\int_{-\infty}^{+\infty} |S'(x)|dx
= \int_{-\infty}^{+\infty}\left|\frac{m'}{m}(x)\right|dx
<+\infty.
\end{equation}
Due to the assumptions on the initial data in Theorem \ref{main-thm}, from $m=m(S)$,  there exist positive constants $k_{ml}$ and $k_{mr}$, such that
\begin{equation}\label{m}
0<k_{ml}<m(x)<k_{mr}.
\end{equation}
Also, we denote positive constants $k_{s}$ and $k_{r}$, then
\begin{equation}\label{sr0}
|s(\cdot,0)|<k_{s}, \quad |r(\cdot,0)|<k_{r}.
\end{equation}
\begin{lemma}\label{lem-sr-bound}
Under the assumptions of Theorem \ref{main-thm}, given a point $(x_1, t_1)$, suppose the solution of \eqref{system-general-pressure}
is $C^1$ in the characteristic triangle bounded by the forward and backward characteristics through
$(x_1, t_1)$ and the line $t=0$. Then, we can prove that
$$
|s(x_1,t_1)|\leq n_s, \quad |r(x_1,t_1)|\leq n_r, \quad i=1,2.
$$
Here $n_s$ and $n_r$ depend on the initial data and the number of 
piecewise monotonic regions.
\end{lemma}
\begin{proof}
Denote the forward and backward characteristics through a point $(x_*, t_*)$ by
\begin{eqnarray*}\begin{split}
&\overrightarrow{\mathcal{L}}_{x_*}:=\{(x, \overrightarrow{t}(x))|x\leq x_*\}
=\{(\overrightarrow{x}(t),t)|t\leq t_*\},\\
&\overleftarrow{\mathcal{L}}_{x_*}:=\{(x, \overleftarrow{t}(x))|x\geq x_*\}
=\{(\overleftarrow{x}(t),t)|t\leq t_*\}.
\end{split}\end{eqnarray*}

First, we will prove this lemma using the following case of three piecewise monotonic regions:  
suppose there is a point $(x_2, t_2)$ in the forward characteristic line and $(x_3,t_3)$ in the backward characteristic line.
Assume  $m'\leq 0$ in the domain where the region from $x=x_2$ to $x=x_3$ intersects with the characteristic triangle, and $m'\geq 0$ in the rest of the characteristic triangle.
						
\begin{figure}[htbp]
	\centering
	\includegraphics[width=8cm]{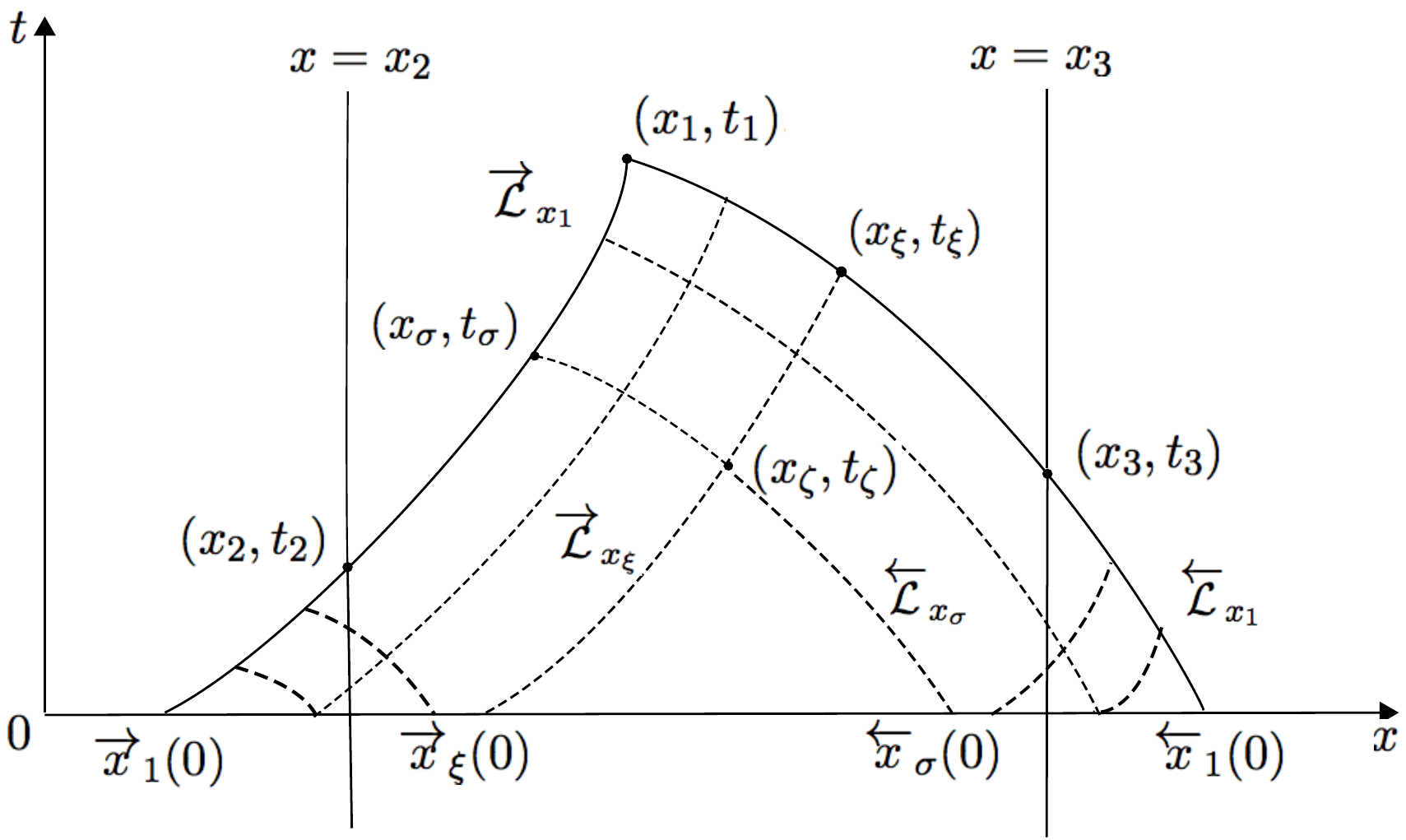}
	\caption{Characteristic triangle} \label{fig-1}
\end{figure}
						
In the forward characteristic line $\overrightarrow{\mathcal{L}}_{x_1}$,
due to $m'(x)\geq 0$ from $(\overrightarrow{x}_1(0),0)$ to $(x_2,t_2)$, so 
integrating \eqref{inequ-s1} along this part, we can get
\begin{eqnarray*}\begin{split}
&s_{11}(x_2,t_2)\geq s_{11}(\overrightarrow{x}_1(0),0)
+\frac{1}{2k_1}\fint_{\overrightarrow{x}_1(0)}^{x_2}\frac{m'}{m}(x)
r_{11}\left(x,\overrightarrow{t}(x)\right)dx,\\
&s_{12}(x_2,t_2)\leq s_{12}(\overrightarrow{x}_1(0),0)
+\frac{1}{2kk_2}\fint_{\overrightarrow{x}_1(0)}^{x_2}\frac{m'}{m}(x)
r_{12}\left(x,\overrightarrow{t}(x)\right)dx.
\end{split}\end{eqnarray*}
Due to the monotone increasing property of  $m(x)$, we have 
$$\left|\frac{m(\overrightarrow{x}_1(0))}{m(x_2)}\right|\leq 1, \quad \quad
\left|\frac{m(x)}{m(x_2)}\right| \leq 1\quad
\mbox{for} ~ x\in\left(\overrightarrow{x}_1(0), x_2\right).
$$
Then, we have
\begin{equation}\label{s1-rightarrow}
|s(x_2,t_2)|\leq |s(\overrightarrow{x}_1(0),0)|
+k_3\fint_{\overrightarrow{x}_1(0)}^{x_2}\left|\frac{m'}{m}(x)\right|
\left|r\left(x,\overrightarrow{t}(x)\right)\right|dx,
\end{equation}
where $k_3=\max\left\{(2k_1)^{-1}, (2kk_2)^{-1}\right\}$.

Since $m'(x)\leq 0$, integrating \eqref{inequ-s2}
from $(x_2,t_2)$ to $(x_1,t_1)$ along the forward characteristic line $\overrightarrow{\mathcal{L}}_{x_1}$, we have
\begin{eqnarray*}\begin{split}
		&s_{21}(x_1,t_1)\geq s_{21}(x_2, t_2)
		+\frac{1}{2kk_1}\fint_{x_2}^{x_1}\frac{m'}{m}(x)r_{21}\left(x,\overrightarrow{t}(x)\right)dx,\\
		&s_{22}(x_1,t_1)\leq s_{22}(x_2, t_2)
		+\frac{1}{2k_2}\fint_{x_2}^{x_1}\frac{m'}{m}(x)r_{22}\left(x,\overrightarrow{t}(x)\right)dx.
	\end{split}\end{eqnarray*}
Recalling \eqref{m}, one can yield
$$
\left|\frac{[m(\overrightarrow{x}_1(0))]}{[m(x_1)]}\right|\leq \frac{k_{mr}}{k_{ml}}, \quad\quad
\left|\frac{[m(x)]}{[m(x_1)]}\right| \leq \frac{k_{mr}}{k_{ml}}.
$$
Considering these two facts,
substitute \eqref{defi-s2r2} into the above two inequalities, we obtain
\begin{equation}\label{s1-rightarrow<}
|s(x_1,t_1)|\leq k_4|s(x_2, t_2)|
+k_5\fint_{x_2}^{x_1}\left|\frac{m'}{m}(x)\right|
\left|r\left(x,\overrightarrow{t}(x)\right)\right|dx,
\end{equation}
where
$k_4=\max\Big\{\left(k_{mr}k_{ml}^{-1}\right)^{\frac{1}{2kk_1}},\left(k_{mr}k_{ml}^{-1}\right)^{\frac{1}{2k_2}}\Big\}$
and
$
k_5=\max\Big\{(2kk_1)^{-1}\big(k_{mr}$
$k_{ml}^{-1}\big)^{\frac{1}{2kk_1}},
(2k_2)^{-1}\left(k_{mr}k_{ml}^{-1}\right)^{\frac{1}{2k_2}}\Big\}
$.

From the above analyses, replacing the integration variable $x$ by $x_\sigma$ in \eqref{s1-rightarrow} and \eqref{s1-rightarrow<} shows:
\begin{eqnarray}\label{s1-rihghtarrow-new}\begin{split}
|s(x_1,t_1)|\leq& k_4 |s(\overrightarrow{x}_1(0),0)|\\
&+(k_3k_4+k_5)\fint_{\overrightarrow{x}_1(0)}^{x_1}\left|\frac{m'}{m}(x_\sigma)\right|
\left|r\left(x_\sigma,\overrightarrow{t}(x_\sigma)\right)\right|dx_\sigma.
\end{split}\end{eqnarray}
As showed in Figure \ref{fig-1}, there are four different cases about the position of 
$\overleftarrow{\mathcal{L}}_{x_\sigma}$, then we will have four different results about 
the relationship between $r(x_\sigma,t_\sigma)$ and $r(\overleftarrow{x}_\sigma(0),0)$
by the above mentioned method. When all these four relsuts are taken into account together, we can get
\begin{eqnarray}\label{r1-leftarrow}\begin{split}
|r(x_\sigma,t_\sigma)|&\leq \max\{1,k_6\} |r(\overleftarrow{x}_\sigma(0),0)|\\
&+\max\{k_6k_7+k_8, k_7+k_8\}
\bint_{\overleftarrow{x}_\sigma(0)}^{x_\sigma}\left|\frac{m'}{m}(x)\right||s(x,\overleftarrow{t}_\sigma(x))|dx,
\end{split}\end{eqnarray}
where
$k_6=\max\Big\{\left(k_{mr}k_{ml}^{-1}\right)^{\frac{1}{2k_1}},\left(k_{mr}k_{ml}^{-1}\right)^{\frac{1}{2kk_2}}\Big\}$,
$k_7=\max\left\{(2kk_1)^{-1},(2k_2)^{-1}\right\}$ and
$k_8=\max\Big\{(2k_1)^{-1}\left(k_{mr}k_{ml}^{-1}\right)^{\frac{1}{2k_1}},
(2kk_2)^{-1}\left(k_{mr}k_{ml}^{-1}\right)^{\frac{1}{2kk_2}}\Big\}$.

Finally, denoting 
$k_9=k_3k_4+k_5$, 
$k_{10} = \max\{1,k_6\}$ and
$k_{11} =\max\{k_6k_7+k_8, k_7+k_8\}$,
by substituting \eqref{r1-leftarrow} into \eqref{s1-rihghtarrow-new}, we can get
\begin{eqnarray}\label{s1r1}\begin{split}
|s(x_1,t_1)|\leq& k_4 |s(\overrightarrow{x}_1(0),0)|
+k_9k_{10}\fint_{\overrightarrow{x}_1(0)}^{x_1}\left|\frac{m'}{m}(x_\sigma)\right||r(\overleftarrow{x}_\sigma(0),0)|dx_\sigma\\
&+k_9k_{11}\fint_{\overrightarrow{x}_1(0)}^{x_1}\left|\frac{m'}{m}(x_\sigma)\right|
\bint_{\overleftarrow{x}_\sigma(0)}^{x_\sigma}\left|\frac{m'}{m}(x)\right||s(x,\overleftarrow{t}_\sigma(x))|dxdx_\sigma.
\end{split}\end{eqnarray}
The first two terms can be bounded by our initial bounds.
Similarly, we also have
\begin{eqnarray}\label{s1r1-new}\begin{split}
|s(x_\xi,t_\xi)|\leq& k_{12} |s(\overrightarrow{x}_\xi(0),0)|
+k_{10}k_{13}\fint_{\overrightarrow{x}_\xi(0)}^{x_\xi}\left|\frac{m'}{m}(x_\zeta)\right||r(\overleftarrow{x}_\zeta(0),0)|dx_\zeta\\
&+k_{11}k_{13}\fint_{\overrightarrow{x}_\xi(0)}^{x_\xi}\left|\frac{m'}{m}(x_\zeta)\right|
\bint_{\overleftarrow{x}_\zeta(0)}^{x_\zeta}\left|\frac{m'}{m}(x)\right||s(x,\overleftarrow{t}_\zeta(x))|dxdx_\zeta,
\end{split}\end{eqnarray}
where $k_{12} = \max\{1, k_4\}$ and $k_{13} = \max\{k_3k_4+k_5, k_3+k_5\}$.
Multiplying \eqref{s1r1-new} by $\frac{m'}{m}(x_\xi)$, and integrating the product from $x_1$ to $\overleftarrow{x}_1(0)$ along
$\overleftarrow{\mathcal{L}}_{x_1}$, we have
\begin{eqnarray}\label{s1r1-m}\begin{split}
&\bint_{x_1}^{\overleftarrow{x}_1(0)}\left|\frac{m'}{m}(x_\xi)\right||s(x_\xi,t_\xi)|dx_\xi\\
&\leq k_{12} \bint_{x_1}^{\overleftarrow{x}_1(0)}\left|\frac{m'}{m}(x_\xi)\right||s(\overrightarrow{x}_\xi(0),0)|dx_\xi\\
&\quad+k_{10}k_{13}\bint_{x_1}^{\overleftarrow{x}_1(0)}\left|\frac{m'}{m}(x_\xi)\right|
\fint_{\overrightarrow{x}_\xi(0)}^{x_\xi}\left|\frac{m'}{m}(x_\zeta)\right||r(\overleftarrow{x}_\zeta(0),0)|dx_\zeta dx_\xi\\
&\quad+k_{11}k_{13}\bint_{x_1}^{\overleftarrow{x}_1(0)}\left|\frac{m'}{m}(x_\xi)\right|
\fint_{\overrightarrow{x}_\xi(0)}^{x_\xi}\left|\frac{m'}{m}(x_\zeta)\right|
\bint_{x_\zeta}^{\overleftarrow{x}_\zeta(0)}\left|\frac{m'}{m}(x)\right|s(x,\overleftarrow{t}_\zeta(x))dxdx_\zeta dx_\xi.
\end{split}\end{eqnarray}
Set
$$
F(x_\eta):=\bint_{x_\eta}^{\overleftarrow{x}_\eta(0)}\left|\frac{m'}{m}(x)\right||s(x,\overleftarrow{t}_\eta(x))|dx.
$$
Since $\overleftarrow{x}_\zeta(0)=\overleftarrow{x}_\sigma(0)$ and $x_\zeta>x_\sigma$
in the same characteristic line, so we have
$F(x_\zeta)\leq F(x_\sigma)$.
Combining with \eqref{v} and \eqref{sr0}, we can rewrite \eqref{s1r1-m} as
$$
F(x_1)\leq k_{12}k_sV+k_{10}k_{13}k_rV^2+k_{11}k_{13}V\fint_{\overrightarrow{x}_1(0)}^{x_1}
\left|\frac{m'}{m}(x_\sigma)\right|F(x_\sigma)dx_\sigma.
$$
Now, using the Gronwall inequality, we can get
$$
F(x_1)\leq \left(k_{12}k_sV+k_{10}k_{13}k_rV^2\right)e^{k_{11}k_{13}V^2}.
$$
For $(x_\sigma,t_\sigma)\in \overrightarrow{\mathcal{L}}_{x_1}$, so $F(x_\sigma)$ is also
bounded by the same quantity. 
Thus, \eqref{s1r1} yields that
\begin{eqnarray*}\begin{split}
|s(x_1,t_1)| \leq k_6k_s+ k_9k_{10}k_rV+k_9k_{11}V
\left(k_{12}k_sV+k_{10}k_{13}k_rV^2\right)e^{k_{11}k_{13}V^2}.
\end{split}\end{eqnarray*}
Similarly, we can get
\begin{eqnarray*}\begin{split}
|r(x_1,t_1)| \leq k_6k_r+ k_9k_{10}k_sV+k_9k_{11}V
\left(k_{12}k_rV+k_{10}k_{13}k_sV^2\right)e^{k_{11}k_{13}V^2}.
\end{split}\end{eqnarray*}
																		
From the above analyses, we can show that the Riemann invariants $s$ and $r$ are bounded in finite piecewise monotic regions.			
\end{proof}

\begin{cor}\label{rmk-bdd}
Under the assumptions of Theorem \ref{main-thm}, we can get 
the $L^\infty$ bounds of $u$ and $h$. Also, we have the upper bound of $\rho$,
$c$ and $p$.
\end{cor}
\begin{proof}
First, \eqref{rs} gives the $L^\infty$ bounds of $u$ and $h$. 
Due to
$$
\int_{\tau}^{1}c(\xi)d\xi = \int_{\tau}^{1}\sqrt{-p_\xi(\xi,S)}d\xi \leq h = \frac{1}{2}(s-r)
$$
and the assumption \eqref{relation-p2},
there exists positive constants
$\tau_{\min}$ and $c_{\max}$ depending only on the initial data such that
$$
\tau(x,t)\geq \tau_{\min}>0, \quad c(x,t)\leq c_{\max}.
$$
So we have the upper bound of $\rho$ and $p$ on account of the definition of $\tau$ and \eqref{p-leq-ch}.
\end{proof}

\section{Singularity formation}
First, we recall the characteristic decompositions. By the definition of $y$ and $q$, we have (\textit{c.f.} \cite{chen-young})
\begin{eqnarray}\label{equ-partial+-partial-q}
\partial_+ y = a_0 + a_1 y - a_2 y^2, \quad\quad
\partial_-q  = a_0 - a_1 q - a_2 q^2, 
\end{eqnarray}
where
\begin{eqnarray*}
	\begin{split}
		& a_0 = -c I_\mu + \frac{\sqrt{c}}{2} \left(\frac{p_\mu}{c}\right)_h p_\mu 
		-c \left(\frac{p_\mu}{c}\right)_h I - \frac{c_h}{2\sqrt{c}}I^2,\\
		& a_1 = -(2\sqrt{c} I)_h,\quad\quad
	 a_2=\frac{c_h}{2\sqrt{c}}>0.
	\end{split}
\end{eqnarray*}

\subsection{Estimate on the root of $a_0+a_1y-a_2y^2=0$}
The first major step is to prove the lower bound on the roots of  
\[
a_0+a_1y-a_2y^2=0,
\]
which is
$$
y_{root} = \frac{-a_1\pm\sqrt{a_1^2 + 4 a_0 a_2}}{-2 a_2}
=\frac{1}{2}\left[\frac{a_1}{a_2} \mp\sqrt{\left(\frac{a_1}{a_2}\right)^2
	+4\frac{a_0}{a_2}}\right].
$$
Here $(\frac{a_1}{a_2})^2+4\frac{a_0}{a_2}\geq0$.
It is noticeable that the lower bound of $y_{root}$ depends on the 
$L^\infty$ estimates of $s$ and $r$.
\begin{lemma}\label{lemma-lower-bound}
Under the assumptions of Theorem \ref{main-thm},
there exists a positive constant $N$ only depending on the initial data such that
\[
\left|\frac{1}{2}\left[\frac{a_1}{a_2} \mp\sqrt{\left(\frac{a_1}{a_2}\right)^2
	+4\frac{a_0}{a_2}}\right]\right|<N.
\]
\end{lemma}
\begin{proof}
We only need to show the boundedness of ${a_1}{a_2}^{-1}$ and ${a_0}{a_2}^{-1}$. 
First, we give elaborative calculation on $\left(p_\mu c^{-1}\right)_h$.
Because 
$$
c_h =  c_\tau \tau_h = -c^{-1} c_\tau
= \frac{1}{2}(-p_\tau)^{-1} p_{\tau\tau},
$$
then
\begin{equation*}
\left(p_\mu c^{-1}\right)_h=c^{-1} p_{\mu h} - c^{-2} c_h p_\mu
= (-p_\tau)^{-\frac{1}{2}} p_{\mu h} - \frac{1}{2} (-p_\tau)^{-2} p_{\tau\tau} p_\mu.
\end{equation*}
Also we have
$$
c^2=-p_\tau=- p_h h_\tau=c p_h, 
$$
this yields $p_h = c$,
$$
p_{\mu h} = p_{h\mu} = c_\mu = -\frac{1}{2} (-p_\tau)^{-\frac{1}{2}} p_{\tau \mu},
$$
which implies,
\begin{equation}\label{equ-pmu-c-1}
\left(p_\mu c^{-1}\right)_h
= -\frac{1}{2} (-p_\tau)^{-1} p_{\tau \mu} - \frac{1}{2} (-p_\tau)^{-2} p_{\tau\tau} p_\mu.
\end{equation}
Then, from the definition $(\ref{equ-h})_2$ of $I$, we can get
\begin{equation}\label{equ-ih}
I_h = \frac{1}{2} \sqrt{c} \left(p_\mu c^{-1}\right)_h
= -\frac{1}{4} (-p_\tau)^{-\frac{3}{4}} p_{\tau\mu}
-\frac{1}{4} (-p_\tau)^{-\frac{7}{4}} p_{\tau\tau} p_\mu.
\end{equation}
Integration by parts yields that
\begin{eqnarray}\label{equ-i-ibp}
\begin{split}
I = \frac{1}{2} \int_{h_0}^h \sqrt{c}d\left(p_\mu c^{-1}\right)
& =\frac{1}{2} \left(c^{-\frac{1}{2}} p_\mu\right)\bigg|_{h_0}^h
-\frac{1}{4} \int_{h_0}^h c^{-\frac{3}{2}} c_h p_\mu dh\\
& = \frac{1}{2} \left[(-p_\tau)^{-\frac{1}{4}} p_\mu \right]\bigg|_{h_0}^h
-\frac{1}{8} \int_{h_0}^h (-p_\tau)^{-\frac{7}{4}} p_{\tau\tau} p_\mu dh.
\end{split}
\end{eqnarray}
Then, we have
\begin{eqnarray}\label{equ-i-mu-ibp}
\begin{split}
I_\mu =&  \frac{1}{2} \left[\frac{1}{4} (-p_\tau)^{-\frac{5}{4}} p_{\tau\mu}p_\mu
+ (-p_\mu)^{-\frac{1}{4}} p_{\mu\mu}\right]\bigg|_{h_0}^h\\
& - \frac{1}{8} \int_{h_0}^h \left[\frac{7}{4}(-p_\tau)^{-2} p_{\tau\mu} p_{\tau\tau} p_\mu + (-p_\tau)^{-\frac{7}{4}} p_{\tau\tau\mu}p_\mu 
+ (-p_\tau)^{-\frac{7}{4}} p_{\tau\tau}p_{\mu\mu}\right]dh.
\end{split}
\end{eqnarray}
Direct calculation shows that
\begin{eqnarray}\label{equ-a1/a2}
\begin{split}
\frac{a_1}{a_2} = -2 I - 4 c c_h^{-1} I_h
=& - 2 I - 8 (-p_\tau)^{\frac{3}{2}} p_{\tau\tau}^{-1} I_h\\
=& - 2 I 
- 2 (-p_\tau)^{\frac{3}{4}} p_{\tau\tau}^{-1} p_{\tau\mu}
+ 2(-p_\tau)^{-\frac{1}{4}} p_\mu.
\end{split}
\end{eqnarray}
and
\begin{eqnarray}\label{equ-a0/a2}
\begin{split}
\frac{a_0}{a_2} =& -2 c^{\frac{3}{2}} c_h^{-1} I_\mu 
+ c c_h^{-1}\left(p_\mu c^{-1}\right)_h p_\mu
- 2 c^{\frac{3}{2}} c_h^{-1} \left(p_\mu c^{-1}\right)_h I
-I^2\\
=& -2 c^{\frac{3}{2}} c_h^{-1}\left[I_\mu + \left(p_\mu c^{-1}\right)_h I\right]
+ c c_h^{-1}\left(p_\mu c^{-1}\right)_h p_\mu
-I^2\\
=& -4 (-p_\tau)^{\frac{7}{4}} p_{\tau\tau}^{-1} 
\left\{I_\mu -\frac{1}{2} \left[(-p_\tau)^{-1} p_{\tau\mu} + (-p_\tau)^{-2} p_{\tau\tau} p_\mu\right]I\right\}\\
& - (-p_\tau)^{\frac{1}{2}} p_{\tau\tau}^{-1} p_{\tau\mu} p_\mu
- (- p_\tau)^{-\frac{1}{2}} p_\mu^2-I^2\\
\end{split}
\end{eqnarray}
By using of  \eqref{assumption-p-ptautau-ptau}, \eqref{inequ-pmu},
\eqref{inequ-ptaumu-pmumu-inequ-ptaotaomu}, \eqref{p-leq-ch} and corollary
\ref{rmk-bdd}, we can get the boundedness of ${a_1}a_2^{-1}$ and 
$a_0a_2^{-1}$.
Therefore, we prove this lemma.
\end{proof}
On the basis of the above lemma, it is easy to get:
\begin{lemma}
Under the assumptions of Theorem \ref{main-thm}, we can prove that
\begin{equation}\label{ineq-yq}
y(x,t)\leq Y, \quad q(x,t)\leq Q,
\end{equation}
where
\begin{equation}\label{defi-YQ}
Y = \max\left\{N, \sup_xy(x,0)\right\},\quad\quad
Q=\max\left\{N, \sup_xq(x,0)\right\}.
\end{equation}
\end{lemma}

\subsection{Time-dependent lower bound on $a_2$}
To show the formation of singularity, the key step is to obtain
the lower bound of $a_2$. In fact, the function $a_2$ might vanish
as time tends to infinity, such as for the gas dynamic case (\textit{c.f.} \cite{CPZ,CPZ2,smoller}).

\begin{lemma}
Assume that the pressure satisfies the assumptions (\textbf{H1}) and (\textbf{H3}),
then
\begin{equation}\label{a2-infinity}
\int_0^\infty a_2 \left(\tau(x,t),t\right)dt = \infty.
\end{equation}
\end{lemma}
\begin{proof}
We know that \eqref{a2-infinity} is true if we can prove
\begin{equation}\label{ineq-a2}
\left[a_2\left(\tau(x,t)\right)\right]^{-1} \leq k_{14} t + k_{15}.
\end{equation}
Direct calculation shows that
$$
a_2 \left(\tau(x,t)\right) = \frac{c_h}{2\sqrt{c}}
=\frac{c_\tau \tau_h}{2\sqrt{c}}
=-\frac{1}{c} \frac{c_\tau}{2\sqrt{c}}
=\frac{1}{4}\left(-p_\tau\right)^{-\frac{5}{4}} p_{\tau\tau}.
$$
Then 
\begin{equation}\label{equ-a2-tau-inverse}
\left[a_2\left(\tau(x,t)\right)\right]^{-1} 
 = 4\left(-p_\tau\right)^{\frac{5}{4}} (p_{\tau\tau})^{-1}.
\end{equation}
From $(\ref{p-tau-k-k-1-relation-A-ptau})_3$, we can get
\begin{equation}\label{ineq-Aptau-ptautau}
\left[4\left(-p_\tau\right)^{\frac{5}{4}} (p_{\tau\tau})^{-1}\right]_\tau
\leq A \left(-p_\tau\right)^{\frac{1}{4}}.
\end{equation}
And we also have
\begin{eqnarray*}
\begin{split}
\left\{\int_{\tau_{\min}}^{\tau(x,t)} \left[-p_\xi(\xi)\right]^{\frac{1}{4}}d\xi\right\}_t
&= \left[\int_{\tau_{\min}}^{\tau(x,t)} \sqrt{c(\xi)}d\xi\right]_t\\
&=\sqrt{c}\tau_t =\sqrt{c} u_x 
=\frac{\sqrt{c}}{2}(s_x+r_x)
=\frac{1}{2}(y+q)\\
&\leq \frac{1}{2}(Y+Q).
\end{split}
\end{eqnarray*}
Integrating the last inequality with respect to $t$ yields
\begin{equation}\label{inequ-k1112}
\int_{\tau_{\min}}^{\tau(x,t)} \left[-p_\xi(\xi)\right]^{\frac{1}{4}}d\xi
\leq \int_{\tau_{\min}}^{\tau(x,0)}\left[-p_\xi(\xi)\right]^{\frac{1}{4}}d\xi
+ \frac{1}{2}(Y+Q)t
\leq k_{16}t + k_{17}.
\end{equation}
Combining \eqref{equ-a2-tau-inverse}, \eqref{ineq-Aptau-ptautau} 
and \eqref{inequ-k1112}, we can complete the proof.
\end{proof}

\subsection{Singularity formation}
In this subsection, we will prove the main theorem.

\textbf{Proof of Theorem \ref{main-thm}:}
We just consider the $\inf\limits_{x\in\R}y(x,0)<-N$ case, the other case for $q$
is similar.

We can assume that $-N$ is a uniform lower bound for the roots of 
$a_0+a_1 y - (1-\eps)a_2 y^2=0$, according to lemma \ref{lemma-lower-bound}.
Then since $a_2>0$, we have
\begin{equation}\label{inequ-equation-eps}
a_0+a_1 y - (1-\eps)a_2 y^2\leq 0, \quad \mbox{for ~every} ~y\leq -N.
\end{equation}
According to the definition of infimum, there exist $0<\eps\ll1$ and $x_0\in\R$ such that
\begin{equation}\label{y-x0-0}
y(x_0, 0) <-(1+\eps)N.
\end{equation}
Now we consider the forward characteristic passing $(x_0,0)$, we have
$$
\partial_+ y\leq -\eps a_2 y^2,
$$
integrating the last inequality from $0$ to $t$ with respect to the time variable, we can get
$$
\frac{1}{y(\overrightarrow{x}(t), t)} \geq \frac{1}{y(x_0,0)} 
+ \eps \fint_0^t a_2(\overrightarrow{x}(t), t) dt,
$$
From \eqref{a2-infinity} and \eqref{y-x0-0}, we can show that $y$  blows up in finite time.

\section*{Acknowledgments}
This research was partially supported by NSFC grant \#11301293/A010801,
and the China Scholarship Council No. 201406210115 as an exchange graduate student at the Georgia Institute of Technology.

\end{document}